\documentclass[a4paper]{amsproc}
\usepackage{amssymb}

\theoremstyle{plain}
 \newtheorem{theorem}{Theorem}[section]
 \newtheorem{lemma}{Lemma}[section]
 \newtheorem{corollary}{Corollary}[section]
\theoremstyle{definition}
 \newtheorem{definition}{Definition}[section]

\renewcommand{\le}{\leqslant}
\renewcommand{\ge}{\geqslant}



\author[Nixon]{\bfseries James Nixon}

\address{
Toronto\\
Canada}
\email{james.nixon@mail.utoronto.ca}

\title[On The Indefinite Sum In Fractional Calculus]{On The Indefinite Sum In Fractional Calculus}

\newcommand{\I}{\frac{1}{2\pi i}\int_{\sigma-i\infty}^{\sigma+i\infty}}
\newcommand{\G}{\Gamma}
\newcommand{\D}[2]{\frac{d^{#2}}{{d#1}^{#2}}}

\newcommand{\T}{\bigtriangledown}
\newcommand{\E}{\mathcal{E}}

\begin{document}

\vspace{18mm} \setcounter{page}{1} \thispagestyle{empty}

\begin{abstract}

We present a theorem on taking the repeated indefinite summation of a holomorphic function $\phi(z)$ in a vertical strip of $\mathbb{C}$ satisfying exponential bounds as the imaginary part grows. We arrive at this result using transforms from fractional calculus. This affords us the ability to indefinitely sum more complicated functions than previously possible; such as holomorphic functions of order $n \in \mathbb{N}$ that have decay at plus or minus imaginary infinity. We then further investigate the indefinite summation operator by restricting ourselves to a space of functions of exponential type. We arrive at a second representation for the indefinite summation operator, equivalent to the first presented, and show we have defined a unique operator on this space. We develop a convolution using the indefinite sum that is commutative, associative, and distributive over addition. We arrive at a formula for the complex iterates of the indefinite sum (the differsum), using this convolution, that resembles the Riemann-Liouville differintegral. We close with a generalization of the Gamma function.

\end{abstract}

\maketitle

\section{Introduction}\label{sec1}

\setcounter{section}{1}
\setcounter{equation}{0}\setcounter{theorem}{0}

Indefinite summation is an operation in mathematics that has a long history. Many mathematicians have investigated the operator, including Ramanujan \cite{rf2} and Euler \cite{rf10}. The reasoning behind defining indefinite summation is to give mathematical sense to the analytic continuation of sums of a holomorphic function. As in, suppose $S(n) = \sum_{j=1}^n f(j)$ for $n \in \mathbb{N}$, what happens if we replace $n$ with $z \in \mathbb{C}$ and require $S$ be holomorphic? There exists techniques for generating the indefinite sum of holomorphic functions and each has its advantages and its flaws, where the majority of the results are accumulated using Bernoulli polynomials and the Laplace transform. In this paper we will distance ourselves from these methods and approach the problem with the differintegral. We will try to explain our best the advantages of our technique. 

The method we introduce for finding the indefinite sum of a holomorphic function employs a triple integral transform that comes from fractional calculus. Although clunky at first sight, many of the results work out cleanly and with fair ease, despite the mess of symbols. The advantage of our method, in comparison to Ramanujan's method, is that we achieve a more localized form of the operator on holomorphic functions; whereas Ramanujan's method requires the function be analytic in a half plane of $\mathbb{C}$, we only require it be analytic in a vertical strip of $\mathbb{C}$. Also, unlike Euler-Maclaurin and Ramanujan summation, our method of performing indefinite summation only requires knowledge of the function on a vertical line in $\mathbb{C}$ (or its values at the positive integers), and none of its derivatives are used in the formula. Our method also allows us to perform an unlimited number of indefinite summations in a single concise formula, which is much more cumbersome using Euler-Maclaurin summation or Ramanujan summation. Further, our method only requires an exponential bound as the imaginary part of the argument grows, this implies we can take the indefinite sum of functions of higher order like $e^{z^2}$ since it has decay at plus or minus imaginary infinity, which is not possible with Ramanujan and Euler-Maclaurin summation.

A large problem that needs to be addressed when investigating the indefinite sum is the uniqueness of a particular indefinite summation operator. By example, we can always add a one-periodic function to the indefinite sum and still get another indefinite sum \cite{rf2}. For this reason an additional criterion is added to the indefinite sum so that it satisfies uniqueness. We do not encounter this problem when defining indefinite summation with our transforms because adding a one-periodic function to our indefinite sum causes our integral expression to blow up at plus or minus imaginary infinity and it no longer converges. This ensures that our indefinite sum is unique, although it satisfies moderately stricter restrictions. 

We elaborate further on a series of results involving the indefinite summation operator that we can arrive at. We define a space we can perform indefinite summation in, and from this we develop a second representation of the indefinite sum equivalent to the first. This representation only requires knowledge of the indefinite sum at the positive integers. From here we are given open ground to talk freely about interesting formulae which arrive due to the recursive behaviour of the indefinite sum. 

We define a convolution involving the indefinite sum. It allows us to express the ``differsum'' in a formula that resembles the Riemann-Liouville differintegral very much. We arrive at these results by only observing the behaviour of the indefinite sum at the non-negative integers by using a convenient and powerful lemma. We close by first iterating the indefinite sum to an arbitrary complex number in the right half plane, and then the indefinite product of the Gamma function. This produces a generalized function that interpolates $z$ and $\G(z)$ while maintaining a recursive pattern.

With these formalities we introduce the indefinite summation operator.

\begin{definition}\label{df1}
If $\phi$ is holomorphic on open $\Omega\subset\mathbb{C}$ then an indefinite sum $\sum$ of $\phi$ is denoted $\sum_z \phi$ and is holomorphic on $\Omega$ and for $z,z+1 \in \Omega$ satisfies, 
$$\sum_z\phi + \phi(z+1) = \sum_{z+1} \phi$$
\end{definition}

We note the strongest result on indefinite summation in complex analysis that we know of, Ramanujan's method, which produces an indefinite sum for any holomorphic function $\phi$ of exponential type $\alpha < 2\pi$ (as in $|\phi(z)| < C e^{\alpha |z|}$) \cite{rf2}. So far we have not defined the indefinite sum quite rigorously, as issues of uniqueness need to be addressed, however, we are satisfied with this formal definition until we bring forth our analytic integral expression for the operator. 

We introduce the backwards difference operator. It appears in combinatorics and recursion quite frequently. Despite its simple definition, considering it for analytic functions proves trickier in comparison to the derivative operator (of which it is a discrete analogue). 

\begin{definition}\label{df2}
If $\phi(z)$ is holomorphic on open $\Omega$ and there is an open subset $\Omega'\subset \Omega$ such that $z \in \Omega'\, \Rightarrow\, z,z-1 \in \Omega$ then the backwards difference $\T_z$ of $\phi$ on $\Omega'$ is given by,
$$\T_z \phi = \phi(z) - \phi(z-1)$$
\end{definition}

The backwards difference operator has appeared in many places in mathematics. The operator has undergone much investigation, but now the concept--like the indefinite sum--in the field of calculus and analysis, exists in a niche.  The intertwining property of the indefinite sum and the backwards difference operator is simple $\sum \T = \T \sum = 1$ or $\sum = \T^{-1}$. Thus, these operators are inverses of each other.

Before progressing we introduce a function that is invaluable to our research, Euler's Gamma function.

\begin{definition}\label{df3}
The Gamma function $\G$ is a meromorphic function on the entire complex plane with poles at the nonpositive integers, for $z \in \mathbb{C},\,\Re(z) > 0$ represented by, 
\begin{equation}\label{Gamma}
\G(z) = \int_0^\infty e^{-t}t^{z-1}\,dt
\end{equation}
\end{definition}
		
		The Gamma function satisfies the functional relationship, $z\G(z) = \G(z+1)$ and converges to the factorial for natural values, $\G(n+1) = n!$. We have imaginary asymptotics of the Gamma function, which are,
		
\begin{equation}\label{asym}
|\G(\sigma + iy)| \sim \sqrt{2\pi} e^{-\sigma-\pi |y|/2} |y|^{\sigma + 1/2}\,\,\,\,\,y \to \pm\infty
\end{equation}

		There exists a plethora of material on the Gamma function and a quick search will give papers unlimited. 

Before introducing our results, we present our differintegral, which will be the hidden tool we will be applying throughout. For our purposes we only define the differintegral centered at zero.

\begin{definition}\label{df4}
Let $f(x)$ be holomorphic in the open sector $\Omega=\{x\in\mathbb{C}\,;\,\pi < \eta < \arg(-x) < \eta' < \pi\,;\,\eta,\eta'\in\mathbb{R},\, \eta'-\eta \le \pi\}$. Assume there exists $b \in \mathbb{R}^+$ dependent on $f$ such that for $\sigma \in \mathbb{R}^+,\,0 < \sigma < b$ and $\theta \in \mathbb{R}, \, \eta < \theta < \eta'$ we have $\int_0^\infty |f(-e^{i\theta}x)|x^{\sigma-1} \,dt < \infty$. Then, the Weyl differintegral centered  at zero $\D{x}{-z}\Big{|}_{x=0}$ of $f$ is represented for $z$, $0 < \Re(z) < b$,
\begin{equation}
\label{diff}
\D{x}{-z}\Big{|}_{x=0} f(x) = \frac{e^{i\theta z}}{\Gamma(z)} \int_{0}^\infty f(-e^{i\theta}x)x^{z-1}\,dx
\end{equation}
\end{definition}

	The differintegral as we've defined it is holomorphic in $z$. We can retrieve $f$ from its differintegral centered at zero by using Mellin's inversion theorem, 

$$f(-e^{i\theta}t) = \I \G(\xi)\Big{(}\D{x}{-\xi}\Big{|}_{x=0} f(x)\Big{)}e^{-i\theta \xi}t^{-\xi}\,d\xi$$

This result shall give us a lot to work with in our arsenal in order to handle the problem of producing an indefinite sum. The last operator we introduce is the Twisted Mellin Transform. This operator presents a strong relationship between differentiating a function and performing the backwards difference of the transformed function. The operator is introduced in \cite{rf11}, however we found it independently and write it in a slightly generalized form.

\begin{definition}\label{df5}
Suppose $f$ is holomorphic in the open sector $\Omega = \{ x\in \mathbb{C} \,;\, -\pi/2 < -\eta < \arg(x) < \eta < \pi/2\,;\,\eta \in \mathbb{R}^+\}$. Assume for $\sigma \in \mathbb{R}^+$ in the strip $b \in \mathbb{R}^+,\,0 < \sigma < b$ and $\theta \in \mathbb{R},\, -\eta < \theta < \eta$ we have $\int_0^\infty e^{-\cos(\theta)x}|f(e^{i\theta}x)|x^{\sigma - 1}\,dx<\infty$. Then, the Twisted Mellin transform $\mathcal{Y}$ of $f$ is represented for $0 < \Re(z) < b$,
$$\mathcal{Y}_zf=\D{x}{-z}\Big{|}_{x=0} e^x f(-x)$$
\end{definition}

The fundamental property of the Twisted Mellin transform is,
\begin{equation}\label{TMT}
\mathcal{Y}_z \frac{df}{dx} = \T_z \mathcal{Y}_z f
\end{equation}

which follows by, 
\begin{eqnarray*}
\D{x}{-z}\Big{|}_{x=0} [e^x f'(-x)] &=& \D{x}{-z}\Big{|}_{x=0} \Big{[}e^x f(-x) - \frac{d}{dx} e^x f(-x)\Big{]}\\
&=& \D{x}{-z}\Big{|}_{x=0} e^xf(-x) - \D{x}{-(z-1)}\Big{|}_{x=0} e^x f(-x)
\end{eqnarray*}

With these introductions we see we have enough to state our analytic continuation of the summation operator. We hope the brevity of our elaboration on the results outside of this paper we shall use does not affect its comprehension.

\section{The Analytic Continuation of the Summation Operator}\label{sec2}

\setcounter{section}{2}
\setcounter{equation}{0}\setcounter{theorem}{0}

We introduce a few lemmas on the absolute convergence of certain modified Mellin transforms. These lemmas shall follow quickly and are applied one after the other to the triple integral transform that will be our method of analytically continuing the summation operator. The first lemma is inspired by the Paley-Wiener theorem on the Fourier transform, where they show holomorphic functions bounded on a horizontal strip in $\mathbb{C}$ by $\frac{1}{1+x^2}$ have a Fourier transform bounded by an exponential in the same horizontal strip \cite{rf9}. Our result appears much more fitting to our differintegral however. We prove the result using the differintegral, but all it takes is some clever re-maneuvering to make it apply on the Mellin transform--which we leave to the interested reader.

\begin{lemma}\label{lma1}
 Let $\phi$ be holomorphic in the strip $0<\Re(z)=\sigma<b$ for $b,\sigma, \in \mathbb{R}^+$ and let $|\lim_{\Re(z) \to 0} \phi(z)| < \infty$. Suppose that $|\phi(z)| < Ce^{\alpha |\Im(z)|}$ for $\alpha \in \mathbb{R}, \,0 \le \alpha < \pi/2$ and $C \in \mathbb{R}^+$. Then $f(x) = \I \G(\xi)\phi(\xi)x^{-\xi}\,d\xi$ is holomorphic in the sector $\alpha - \pi/2 < \arg(x) < \pi/2 - \alpha$ and $\int_0^\infty |f(e^{i\theta}x)|x^{\sigma - 1}\,dx < \infty$ for $\theta \in \mathbb{R},\,\alpha - \pi/2 < \theta < \pi/2 - \alpha$.

Conversely, let $f$ be holomorphic in the sector $\alpha \in \mathbb{R},\,\alpha-\pi/2 < \arg(x) < \pi/2 - \alpha$.  Suppose for $\sigma \in \mathbb{R}^+$ in the strip $b \in \mathbb{R}^+,\,0 <\sigma < b$ and $\theta \in \mathbb{R},\,\alpha - \pi/2 < \theta < \pi/2 - \alpha$ we have $\int_0^\infty |f(e^{i\theta}x)|x^{\sigma-1} < \infty$. Then $\phi(z) = \frac{1}{\G(z)} \int_0^\infty f(x)x^{z-1}\,dx$ is holomorphic in the strip $0 < \Re(z) < b$ and $|\phi(z)| < M e^{\alpha|\Im(z)|}$ for $M \in \mathbb{R}^+$.
\end{lemma}

\begin{proof}
We prove the first direction. Since $|\phi(z)| < Ce^{\alpha |\Im(z)|}$, for $\alpha - \pi/2 < \theta < \pi/2 - \alpha$ and $x \in \mathbb{R}^+$, $f(e^{i\theta} x) = \I \G(\xi)\phi(\xi)e^{-i\theta \xi} x^{-\xi}\,d\xi$. $f$ is holomorphic because $\frac{1}{2 \pi i} \int_{\sigma - in}^{\sigma + in} \G(\xi)\phi(\xi) x^{-\xi}\,d\xi \to \I \G(\xi)\phi(\xi) x^{-\xi}\,d\xi$ uniformly in $x$ in the open sector. We also know that $|f(e^{i\theta} x)| < \frac{x^{-\sigma}}{2 \pi} \int_{\sigma - i \infty}^{\sigma + i \infty} |\G(\xi)\phi(\xi)e^{-i\theta \xi}|\,dz = C_\sigma x^{-\sigma}$ where the integral is absolutely convergent again because of the bounds on $\phi$ and the asymptotics of the Gamma function \eqref{asym}. Now we can shift $\sigma$ in the open strip and this does not affect convergence so that $|f(e^{i\theta}x)| < Cx^{-\epsilon}$ as $x \to 0$ for every $\epsilon > 0$ and $|f(e^{i\theta}x)| < Cx^{-b}$ as $x \to \infty$ for $ C \in \mathbb{R}^+$ and $C = \sup_{\sigma \in [\epsilon,b]} \{\frac{1}{2\pi}\int_{\sigma - i\infty}^{\sigma + i\infty} |\G(\xi) \phi(\xi) e^{-i\theta \xi}|\,d\xi\}$. This ensures that $\int_0^\infty |f(e^{i\theta}x)|x^{\sigma -1 }\,dx < \infty$ for $0 < \sigma < b$ and $\alpha - \pi/2 < \theta < \pi/2 - \alpha$.

For the other direction we have $$\G(z)\phi_\theta(z) = \int_0^\infty f(e^{i\theta} x) x^{z-1}\,dx =  e^{-i\theta z}\int_0^\infty f(x)x^{z-1}\,dx = e^{-i\theta z}\G(z)\phi_0(z)$$ which follows by contour integration. Now take the absolute value to see that $|\phi_0(z)| < C_\sigma\frac{ e^{\theta \Im(z)}}{|\G(z)|}$ where $\alpha -\pi/2< \theta < \pi/2 - \alpha$ and $C_\sigma = \int_0^\infty |f(e^{i\theta}x)|x^{\sigma - 1}\,dx$. By the asymptotics of the Gamma function \eqref{asym}, $|\phi_0(z)| < M e^{(\pi/2 - \theta)|\Im(z)|}$ which by taking $|\theta| = \pi/2 - \alpha$ we get $|\phi_0(z)| < Me^{\alpha|\Im(z)|}$, where $\phi_0 = \phi$ and $M \in \mathbb{R}^+$. We note we have dropped the term $|y|^{-\sigma-1/2}$ in the inverse Gamma function's asymptotics because it decays to zero as $y \to\pm \infty$. 
\end{proof}

What we have just shown is subtle but very useful. We have not seen it proved before and so we decided to show a proof--though we have our suspicions a proof is already out there. It shows we have an isomorphism between spaces of holomorphic functions (the space $f$ is in to the space $\phi$ is in) which is something we can manipulate freely. Our next lemma comes from asymptotics for the incomplete Gamma function and is the core of the convergence of our triple integral.

\begin{lemma}\label{lma2}
Let $f(x)$ be holomorphic in the sector $\Omega = \{x \in \mathbb{C}\,;\, \alpha - \pi/2 < \arg(x) < \pi/2 -\alpha\}$. Suppose for $\sigma\in \mathbb{R}^+$ in the strip $b\in \mathbb{R}^+,\,0 < \sigma < b$ and $\theta \in \mathbb{R},\,\alpha-\pi/2 < \theta <\pi/2 - \alpha$ we have $\int_0^\infty |f(e^{i\theta}x)|x^{\sigma - 1}\,dx < \infty$. Then for $\lambda \in \mathbb{R}^+$ and $0 < \sigma < b$, $$\int_0^\infty e^{-\lambda t}t^{\sigma-1} \Big{(}\int_0^t e^{ \lambda x} |f(e^{i\theta}x)|\,dx \Big{)}\,dt < \infty$$
\end{lemma}

\begin{proof}
Let us first note that $f(0) < C$ so that $\int_0^t e^x f(x) \,dx < \infty$. Perform integration by parts:

$$\int_0^\infty e^{- \lambda t}t^{\sigma - 1}\int_0^t e^{\lambda x} |f(e^{i\theta}x)|\,dx\,dt = \int_0^\infty e^{\lambda x} |f(e^{i\theta}x)| \int_x^\infty e^{- \lambda t}t^{\sigma - 1}\,dt\,dx$$

From looking at \cite{rf3} we have an asymptotic expansion of the incomplete Gamma function in the integral, $\int_x^\infty e^{-\lambda t}t^{\sigma - 1}\,dt \sim \Big{(}\frac{x}{\lambda}\Big{)}^{\sigma}e^{-\lambda x} \sum_{k=0}^\infty \frac{b_k}{(\lambda x-\sigma)^{k+1}}$ which implies we can find a $C \in \mathbb{R}^+$ such that $\int_x^\infty e^{-\lambda t}t^{\sigma - 1}\,dt < C e^{-\lambda x}x^{\sigma - 1}$. Where this follows because we can factor $1/x$ from the sum and the rest of the series is $\mathcal{O}(1)$ as $x \to \infty$. This shows,

$$\int_0^\infty e^{\lambda x} |f(e^{i\theta}x)| \int_x^\infty e^{- \lambda t}t^{\sigma - 1}\,dt\,dx < C \int_0^\infty|f(e^{i\theta}x)|x^{\sigma-1}\,dx$$

This gives the result.
\end{proof}

We can now give our theorem on the analytic continuation of the summation operator.  We may not have given enough motivation as to why the integral transforms we present are the ones we are looking for, but the recursive behaviour is satisfied and from a familiarity with fractional calculus it seems less questionable and more apparent. The expression may seem to have been plucked from the air since we have tried to separate the direct attachment to the differintegral, however someone who is familiar with the transforms will clearly see that we are conjugating (in terms of abstract algebra) the integral operator by the Twisted Mellin transform, where the inverse Twisted Mellin transform is calculated using the inverse Mellin transform.

\begin{theorem}\label{thm1}
Let $\phi(z)$ be holomorphic in the strip $b \in \mathbb{R}^+,\,0 < \Re(z) < b$, $b > 1$ and let $|\lim_{\Re(z) \to 0} \phi(z)| < \infty$. Let $|\phi(z)| < C e^{\alpha|\Im(z)|}$ for $\alpha \in \mathbb{R},\,0 \le \alpha < \pi/2$ and $C \in \mathbb{R}^+$. Then the indefinite sum of $\phi$, $\sum_z \phi$ is represented for $0 < \sigma < b$ and $0 < \Re(z) < b$ by,
\begin{equation}\label{Indfsum}
\sum_z \phi  = \frac{1}{\G(z)} \int_0^\infty e^{-t}t^{z-1} \Big{(}\int_0^t \frac{e^x}{2 \pi i} \Big{(} \int_{\sigma - i \infty}^{\sigma + i\infty} \G(\xi)\phi(\xi)x^{-\xi}\,d\xi \Big{)}\,dx \Big{)} \,dt
\end{equation}
or written more compactly,
\begin{equation}\label{Indfsum2}
\sum_z \phi  = \mathcal{Y}_z \int_0^x \mathcal{Y}_x^{-1} \phi
\end{equation}
\end{theorem}

\begin{proof}
We first show convergence of the triple integral. This will not be too difficult as the bulk of the work was placed in the lemmas before. If $f(x) = \frac{1}{2 \pi i}\int_{\sigma - i \infty}^{\sigma + i\infty} \G(\xi)\phi(\xi)x^{-\xi}\,d\xi$ it is holomorphic in the sector $\alpha -\pi/2 < \arg(x) < \pi/2 - \alpha$ and satisfies the conditions of Lemma \ref{lma1}, so we know, $\int_0^\infty |f(e^{i\theta}x)|x^{\sigma - 1}\,dx < \infty$ for $\sigma \in \mathbb{R}^+$ in the strip $0 < \sigma < b$. Therefore $f$ satisfies the conditions of Lemma \ref{lma2}, which implies $\int_0^\infty e^{-t} t^{\sigma - 1} \int_0^t e^{x}|f(e^{i\theta}x)|\,dx\,dt < \infty$. Therefore the triple integral transform is convergent. Also, $\int_0^\infty e^{-\cos(\theta)t} t^{\sigma - 1} |\int_0^{e^{i\theta}t} e^{e^{i\theta}x}f(x)\,dx|\,dt < \infty$ because all of these transforms converge in sectors, so in the last step we are justified in applying the Twisted Mellin transform. 

Now we show that this is in fact an indefinite sum of $\phi$. Let us assume that $1< \Re(z)< b$ so that $\T_z \sum_z \phi$ exists and we get, by applying Equation \eqref{TMT},

\begin{eqnarray*}
\T_z \sum_z \phi  &=& \T_z \Big{[}\frac{1}{\G(z)} \int_0^\infty e^{-t}t^{z-1} \int_0^t e^x f(x)\,dx\,dt \Big{]}\\
&=& \frac{1}{\G(z)} \int_0^\infty e^{-t}t^{z-1} \frac{d}{dt} \Big{(}\int_0^t e^x f(x)\,dx \Big{)}\,dt\\
&=& \frac{1}{\G(z)} \int_0^\infty e^{-t}t^{z-1}e^t f(t)\,dt\\
&=& \frac{1}{\G(z)} \int_0^\infty  f(t)t^{z-1}\,dt\\
&=& \phi(z)\\
\end{eqnarray*}

Therefore we know that this operator satisfies $\T \sum = 1$. Let us now show it happens on the right as well. We first note that $\frac{1}{2 \pi i} \int_{\sigma - i\infty}^{\sigma + i \infty} \G(\xi) (\phi(\xi) - \phi(\xi-1))x^{-\xi}\,d\xi = f(x) + f'(x)$ and that $\int_0^t e^x (f(x) + f'(x))\,dx = e^t f(t) - f(0)$,

\begin{eqnarray*}
\sum_z (\T_z\phi(z))  &=& \frac{1}{\G(z)} \int_0^\infty e^{-t}t^{z-1} \int_0^t e^x (f(x)+f'(x))\,dx\,dt\\
&=& \frac{1}{\G(z)} \int_0^\infty (f(t)-f(0)e^{-t})t^{z-1}\,dt\\
&=&\phi(z) - f(0)\\
&=& \phi(z) + A\
\end{eqnarray*}

Therefore we have shown $\T \sum \phi = \phi$ and $\sum \T \phi = \phi + A$ (where $A \in \mathbb{C}$, and this constant is much like a constant of integration, as in $\int_c^t g'(t) \,dt = g(t) + A$) so that it is natural to say $\T^{-1} = \sum$, which is a defining property of $\sum$ and assures that $\sum_z \phi + \phi(z+1) = \sum_{z+1} \phi$. This gives the result.
\end{proof}

More advantageously we can show that $\sum_z$ is exponentially bounded as its imaginary argument grows so that we may consider $\sum_z \sum_z \phi$. Going on in such a manner we can produce all the natural iterates of the indefinite summation operator.

\begin{lemma}\label{lma3}
Let $\phi(z)$ be holomorphic in the strip $b \in \mathbb{R}^+,\,0 < \Re(z) < b$ for $b > 1$ and let $|\lim_{\Re(z) \to 0} \phi(z)| < \infty$. Assume $|\phi(z)|< C e^{\alpha |\Im(z)|}$ for $\alpha \in \mathbb{R},\,0 \le \alpha < \pi/2$ and $C \in \mathbb{R}^+$. Then, $|\sum_z \phi| < M e^{\alpha |\Im(z)|}$ for $M \in \mathbb{R}^+$.
\end{lemma}

\begin{proof}
Observe the representation of the indefinite sum in \eqref{Indfsum} and observe that the final integral transform applied is the Mellin transform. Observe further, if $f(x) = \frac{1}{2 \pi i} \int_{\sigma - i\infty}^{\sigma + i \infty} \G(\xi)\phi(\xi)x^{-\xi}\,d\xi$ we know $f$ is holomorphic in the sector $\alpha - \pi/2 < \arg(x) < \pi/2 - \alpha$. If $\alpha - \pi/2 < \theta < \pi/2 - \alpha$ then, $\int_0^\infty t^{z-1}e^{-t}\int_0^t e^x f(x)\,dx\,dt = e^{i\theta z}\int_0^\infty t^{z-1}e^{-e^{i\theta}t}\int_0^t e^{e^{i\theta}x} f(e^{i \theta}x)\,dx\,dt$. We have done a substitution through both integrals which is justified since all these transforms converge in sectors. We know that $\theta$ is arbitrary and the final function does not depend on $\theta$ by considering contour integrals. We leave the details to be filled by the reader.  These integrals are absolutely convergent, again by Lemma \ref{lma1} and Lemma \ref{lma2}. By Lemma \ref{lma2} we also see that 

\begin{eqnarray*}
|\int_0^\infty t^{z-1}e^{-e^{i\theta}t}\int_{0}^t e^{e^{i\theta}x} f(e^{i \theta}x)\,dx\,dt| &\le& \int_0^\infty t^{\sigma-1}e^{-t\cos(\theta)}\int_{0}^t e^{x\cos(\theta)} |f(e^{i\theta}x)|\,dx\,dt\\ 
&<& K \int_0^\infty |f(e^{i\theta}x)|x^{\sigma - 1}\,dx = C_\sigma
\end{eqnarray*}

 for  some $K\in \mathbb{R}^+$. This shows that, $$\sum_z \phi =\frac{e^{i \theta z}}{\G(z)} \int_0^\infty e^{-e^{i\theta}t} t^{z-1} \int_0^t e^{e^{i\theta}x} f(e^{i \theta}x)\,dx \,dt$$ so we have $|\sum_z \phi | < C_\sigma e^{(\pi/2 - \theta)|\Im(z)|}$ which gives us $|\sum_z \phi| < M e^{\alpha |\Im(z)|}$ for an $M \in \mathbb{R}^+$. This shows the result.
\end{proof}

With this we can state our corollary on the fact we can take any natural iterate of the indefinite sum and it will remain holomorphic in the same strip.

\begin{corollary}\label{cor1}
If $\phi(z)$ is holomorphic in the strip $b \in \mathbb{R}^+,\,0 < \Re(z) < b$ for $b>1$ and $|\phi(z)| < C e^{\alpha |\Im(z)|}$ for $\alpha \in \mathbb{R},\,0 \le \alpha < \pi/2$ and $C \in \mathbb{R}^+$ and $|\lim_{\Re(z) \to 0} \phi(z)| < \infty$ then $\phi$ can be indefinitely summed an arbitrary amount of times
\begin{equation}\label{rptsum}
\sum_z ... (n \, times)...\sum_z \phi = \mathcal{Y}_z \int_0^x ...(n \, times)...\int_0^x \mathcal{Y}^{-1}_x \phi
\end{equation}
\end{corollary}

\begin{proof}
Observe Theorem \ref{thm1} and Lemma \ref{lma3}. Theorem \ref{thm1} ensures we can indefinitely sum $\phi$ and then Lemma \ref{lma3} ensures that that indefinite sum satisfies the original conditions of Theorem \ref{thm1}. The representation \eqref{rptsum} follows by induction and by applying Lemma \ref{lma2} repeatedly.
\end{proof}

Now before ending this section and conversing solely of functions of exponential type, we see that functions like $e^{z^2}$ as well as $e^{z^3}$ are indefinitely summable. We can find more complicated functions of order greater than one that are indefinitely summable if they behave nicely as the imaginary argument grows.

\section{The exponential space $\E$ and the indefinite sum}\label{sec3}

\setcounter{section}{3}
\setcounter{equation}{0}\setcounter{theorem}{0}

As we've defined the indefinite sum at this point, we've only used functions bounded by an exponential as the imaginary part grows. We now consider bounding the function by an exponential as the real part of $z$ grows as well. For this reason, we define an exponential space in which the indefinite sum operates. The exponential bounds will ensure we have a unique operator on the space and will show our second representation as well as our first representation are equivalent where they intersect in definition.

\begin{definition}\label{df7}
Suppose $f(z)$ is a holomorphic function on the right half plane $z\in \mathbb{C},\,\Re(z) > 0$ and let $|\lim_{\Re(z) \to 0} f(z)| < \infty$. Assume there exists $\alpha,\rho ,C\in \mathbb{R}^+,\,0\le\alpha < \pi/2,$ such that $|f(z)| < C e^{\alpha |\Im(z)|+\rho|\Re(z)|}$. The space $\E$ contains all such functions $f$ that satisfy these conditions.
\end{definition}

$\E$ contains polynomials, exponentials with a base that have a real part greater than zero, rationals with poles in the left half plane, and other special functions. Aside, on notation, we will write $\alpha_f$ for the bound of $f$ as the imaginary part grows and we will write $\rho_f$ for the bound of $f$ as the real part grows (when distinction between functions is required). For our purposes we will be using the indefinite sum we defined in Theorem \ref{thm2}. Let $f$ belong to $\E$ and let $\sigma\in\mathbb{R}^+,\,\sigma > 0$,

\begin{equation}
\sum_z f = \frac{1}{\G(z)} \int_0^\infty t^{z-1} e^{-t} \int_0^{t} \frac{e^x}{2 \pi i} \int_{\sigma-i\infty}^{\sigma + i \infty} \G(\xi)f(\xi)x^{-\xi}\,d\xi\,dx\,dt
\end{equation} 

We know that if $|f(z)| < C_{\Re(z)} e^{\alpha |\Im(z)|}$ that this representation converges and that $|\sum_z f| < M_{\Re(z)} e^{\alpha |\Im(z)|}$. Therefore $\sum$ is defined for any element of $\E$. We also have the convenient formula, $$\sum_{n} f = \sum_{j=1}^n f(j)$$ which follows by the defining recursion of $\sum$ and because $\sum_0 f = 0$ since $\lim_{z \to 0} \D{x}{-z}\Big{|}_{x=0} e^x \int_0^x e^t g(t) \,dt = e^x \int_0^x e^t g(t) \,dt \Big{|}_{x=0}=0$. This limit is allowed to be taken because $|e^{-x} \int_0^x e^t g(t) \, dt |< C_{\epsilon}x^{1-\epsilon}$ for any $\epsilon > 0$ as $x\to 0$ so that the integral expression for the differintegral converges at $z=0$. We show now that $\sum$ sends $\E \to \E$.

\begin{theorem}\label{thm2}
Suppose $f \in \E$, then $\sum f \in \E$.
\end{theorem}

\begin{proof}
 Let us say $|f| < C e^{\alpha|\Im(z)| + \rho|\Re(z)|}$. We know that $\sum_z f$ is bounded by an exponential and of type less than $\pi/2$ on the imaginary line. Now observe that,

$$\sum_z f = \sum_{z- \lfloor \Re(z) \rfloor} f + \sum_{j=1}^{\lfloor \Re(z) \rfloor} f(z-\lfloor \Re(z) \rfloor + j)$$

which comes from the defining property of $\sum$ and shows that,

$$|\sum_z f| < \sum_{j=0}^{\lfloor \Re(z) \rfloor} Me^{\alpha |\Im(z)| + \rho|\Re(z)-\lfloor \Re(z) \rfloor + j|} < M|\lfloor\Re(z)\rfloor + 1|e^{\alpha|\Im(z)| + \rho|\Re(z)|}$$

And therefore since $|\lfloor\Re(z)\rfloor + 1| < C e^{\epsilon |\Re(z)|}$ for some $\epsilon > 0$ this shows the result.

\end{proof}

With this we are now able to talk freely of the indefinite sum acting on a space. We can also give a more compact form for the operator which follows by Ramanujan's master theorem. 

\begin{corollary}\label{cor2}
Let $f \in \E$ and define the entire function $\vartheta(x) = \sum_{n=0}^\infty \Big{(}\sum_{j=1}^{n+1} f(j) \Big{)}\frac{x^n}{n!}$. Then for $z \in \mathbb{C},\,\Re(z) > 0$, $$\sum_{z} f = \frac{1}{\G(1-z)}\Big{(}\sum_{n=0}^\infty \big{(} \sum_{j=1}^{n+1} f(j) \big{)} \frac{(-1)^n}{n!(n+1-z)} + \int_1^\infty \vartheta(-x)x^{-z}\,dx\Big{)}$$
\end{corollary}

With these results we are prepared to analyze the indefinite sum in this space more intricately.

\section{On Indefinite Sum Convolution}\label{sec4}

\setcounter{section}{4}
\setcounter{equation}{0}\setcounter{theorem}{0}

Beginning this section we will prove a lemma that appears elsewhere but we show it using techniques from fractional calculus. This result will be applied repeatedly on various transforms involving the indefinite sum which will give us a powerful way of proving properties of the indefinite sum over the complex plane by looking at its values on the positive integers.

\begin{lemma}\label{lma4}
Let $f \in \E$ and for $n \in \mathbb{N}$ assume $f(n) = 0$, then $f= 0$.
\end{lemma}

\begin{proof} Take $0 < \sigma <1$ and define $g(x) = \frac{1}{2\pi i } \int_{\sigma - i\infty}^{\sigma + i \infty} \G(\xi) f(1-\xi) x^{-\xi}\,d\xi$. By a simple exercise in contour integration we have $g(x) = \sum_{n=0}^\infty f(n+1) \frac{x^n}{n!} = 0$ this implies that $f=0$. 
\end{proof}

We see this result follows quickly, however its power is abundant. Firstly, it implies if $f_1 (n) = f_2 (n)$ for $n \in \mathbb{N}$ then $f_1(z) = f_2(z)$ for all $z \in \mathbb{C}_{\Re(z) \ge 0}$. We will not waste time in putting it to use. With this we define a new convolution that is associative, commutative, and spreads across addition. To be clear, we will make the following statement on notation, $\sum_z f = \sum_z f(s) \,\T_s$, where the additional terms declare the variable we are performing the sum across. This notation allows us to nest more variables which is required for what comes next.

\begin{definition}\label{df8}
Suppose $f(z-1), g(z-1) \in \E$ and $\alpha_f + 2\alpha_g < \pi/2$, then $$f\times g = \sum_{z+1} f(s-1)g(z+1-s)\,\T_s$$.
\end{definition}

We note this is definitively a convergent expression if we fix $z$ and take the indefinite sum from Theorem \ref{thm1}. The motivation for introducing this convolution is because $(f \times g) (n) = \sum_{j=0}^n f(j) g(n-j)$ which is quite a familiar expression. Furthermore we know that the space $\E$ is almost completely determined by how its functions behave on the naturals so this is bound to be valuable in some sense. We show some properties of this convolution acting in the space $\E$.

\begin{theorem}\label{theom3}
Suppose $f,g,h \in \E$ then:
\begin{enumerate}
\item If $\alpha_f+ 2\alpha_g < \pi/2$, then if $f \times g = h$ we have $h \in \E$.
\item If $\alpha_f+ 2\alpha_g < \pi/2$ and $2\alpha_f+ \alpha_g < \pi/2$, then  $f \times g = g \times f$.
\item If $\alpha_f+ 2(\alpha_g+ 2\alpha_h)< \pi/2$ and  then $f \times (g \times h) = (f \times g) \times h$.
\item If $\alpha_f+  2 \alpha_g < \pi/2$ and $\alpha_f+  2 \alpha_h < \pi/2$ then $f \times (g + h) = (f \times g) + (f \times h)$.
\item If $\alpha_f+ 2 \alpha_g < \pi/2$ then $f \times g = 0$ implies that $f=0$ or $g=0$.
\item Let $(z+1)_{s-1} = \frac{\G(z+s)}{\G(z+1)}$ for $s \in \mathbb{C}_{\Re(z) > 0}$ and $s_1, s_2 \in \mathbb{C}_{\Re(z) > 0}$ we have  $\frac{1}{\G(s_1)}(z+1)_{s_1-1} \times \frac{1}{\G(s_2)}(z+1)_{s_2-1} = \frac{1}{\G(s_1 + s_2)}(z+1)_{s_1+s_2-1}$.
\end{enumerate}
\end{theorem}
\begin{proof}

We prove each result one by one.

\begin{enumerate}

\item First take $f\times g = \sum_{z-\lfloor\Re(z)\rfloor} f(s-1)g(z+1-s)\,\T_s + \sum_{j=1}^{\lfloor \Re(z) \rfloor} f(z-\lfloor \Re(z) \rfloor-1+j)g(\lfloor \Re(z) \rfloor + 1 - j)$. Therefore since $|f(s-1)g(z + 1-s)| < C e^{\alpha_f |\Im(s)| + \alpha_g(|\Im(s)| + |\Im(z)|) + (\rho_f + \rho_g)|\Re(s)|}$ we see that $|\sum_{z} f(s-1) g(z +1 - s) \,\T_s| < C e^{(\alpha_f + 2\alpha_g)|\Im(z)| + (\rho_f + \rho_g + \epsilon)|\Re(z)|}$. Where we treated $C e^{\alpha_g |\Im(z)|}$ as a constant independent of the variable $s$ we are summing across pulled through the integral.

\item For the second, observe $f \times g$ and $g \times f$ are defined and if $n\in \mathbb{N}$, that $(f \times g)(n) = (g \times f)(n)$, so the result follows by Lemma \ref{lma4}.

\item For the third, by Lemma \ref{lma4} the result follows.

\item For the fourth, we observe the linearity of the indefinite sum and multipication.

\item For the fifth, we use Lemma \ref{lma4} again. As in $\sum_{j=0}^n f(j)g(n-j) = 0$ for all $n$ implies $f(n) = 0$ or $g(n) = 0$ for all $n$.

\item For the sixth case, observe that $f \times 1 = \sum_{z+1} f(s-1)\, \T_s$ so that $\frac{1}{(n-1)!}(z+1)_{n-1} = 1 \times 1 \times 1 \times ...(n\, times)... \times 1$. By associativity the result follows for integers so that $\frac{1}{(n-1)!}(z+1)_{n-1} \times \frac{1}{(m-1)!}(z+1)_{m-1} = \frac{1}{(n+m-1)!}(z+1)_{n+m-1}$. Now apply Lemma \ref{lma4} on $s_1$ and $s_2$ since $\frac{\G(z+s_{12})}{\G(z+1)\G(s_{12})} \in \E$ in $s_1,s_2$ when $\Re(z),\Re(s_1),\Re(s_2) > 0$.  The result follows from this.
\end{enumerate}

This shows all the cases.
\end{proof}

With this we are prepared to show a new representation of complex iterations of the indefinite sum. To this end we define a new operator called the differsum. This term is novel to us and is used to express the operators similarity to the differintegral.

\begin{definition}\label{df9}
Let $\phi(z)$ be holomorphic in the open set $\Omega$.  A differsum $\T^s_z$ of $\phi$ is holomorphic in $s$ on open $\Omega'$ and holomorphic in $z$ for $z\in \Omega$, satisfying:
\begin{enumerate}
\item For $n \in \mathbb{N}$ we have $\T_z^n \phi = \T_z \T_z \dots (n\,times) \dots \T_z \phi$ and $\T_z^{-n} \phi = \sum_z \sum_z \dots (n\,times) \dots \sum_z \phi$.
\item For $s_0,s_1, s_0 + s_1 \in \Omega'$, $\T_z^{s_0}\T_z^{s_1} \phi = \T_z^{s_0 + s_1} \phi$.
\end{enumerate}
\end{definition}

The definitive properties of the differsum is what we would expect of an iterate of the backwards difference operator/indefinite sum. It appears as an object that would be very difficult to analyze. Thankfully though, the deep connection between recursion and our differintegral allows us to speak freely of the differsum. 

\begin{theorem}\label{theom4}
Suppose $f(z) \in \E$ and $q \in \mathbb{C},\,\Re(q) > 0$, then $\T_{z}^{-q} f(z) = \frac{1}{\G(q)} \sum_{z} f(s) (z+1-s)_{q-1}\,\T_s$.
\end{theorem}

\begin{proof}
We make use of the fact $\T_{z+1}^{-q} f(z-1) = \frac{1}{\G(q)} (f(z) \times (z+1)_{q-1}) (z)$, where this is well defined because $f \in \E$ and because $(z+1)_{q-1} \in \E$ in $z$ and as the imaginary part of $z$ grows it is bounded by $M_\epsilon e^{\epsilon |\Im(z)|}$ for any $\epsilon > 0$. The proof of this is by induction for $q$ an integer. It is clearly true for $q=n=1$. Assume for $n$. Since $\frac{1}{(n-1)!}(z+1)_{n-1} \times 1 = \frac{1}{n!}(z+1)_{n}$ and since $\T_{z+1}^{-(n+1)} f(z-1) =\T_{z+1}^{-n} f(z-1) \times 1 = (f(z) \times \frac{1}{(n-1)!}(z+1)_{n-1}) \times 1 = f(z) \times (\frac{1}{(n-1)!}((z+1)_{n-1} \times 1)) = \frac{1}{n!}f(z) \times (z+1)_n$. Now since, for $q_1, q_2 \in \mathbb{C}$, we have $\frac{1}{\G(q_1)}(z+1)_{q_1-1} \times \frac{1}{\G(q_2)}(z+1)_{q_2-1} = \frac{1}{\G(q_1 + q_2)}(z)_{q_1+q_2-1}$ we know that $\T_{z+1}^{-q_1} \T_z^{-q_2} f(z-1) = \T_{z+1}^{-(q_1+q_2)} f(z-1)$. Therefore removing the transfers up in $z$ we have for $\Re(q) > 0$,

$$ \T_z^{-q} f(s) = \frac{1}{\G(q)} \sum_z f(s) (z+1-s)_{q-1}\,\T_s$$
\end{proof}

To close we make use of our expression for the differsum. We take the transform on $\log(1+s)$. This will produce a generalization of the Gamma function. This comes about in the following sense, $\sum_z \log(1+s)\,\T_s = \log(\Gamma(1+z))$, which follows by Lemma \ref{lma4}. Exponentiate and we have a well defined expression for $\Gamma$. Furthermore the iterates of the indefinite sum once exponentiated will produce functions $\G_0(z) = z, \G_1 (z) = \G(z), \G_2(z), \G_3(z), ...$ such that $\G_n(z) \G_{n+1}(z) = \G_{n+1}(z+1)$. These functions are normalized at one so that $\G_n(1) = 1$.

We wish to be more brazen than this however and to develop the complex iterates of this recursion. Namely for $z,q \in \mathbb{C}$ we have if $\Re(q) > 0$ and $\Re(z) > 1$ that the function 
$$\G_q(z) = e^{\frac{1}{\G(q)} \sum_{z-1} \log(1+s) (z-s)_{q-1} \, \T_s}$$
satisfies the recursion $\G_q(z) \G_{q+1}(z) = \G_{q+1}(z+1)$. This function expands the $\Gamma$ function considerably and allows for a broad generalization of the Barnes G-function. In spirit, $\G_2$ satisfies the same recursion as the Barnes G-function--we are unsure if the two agree however.

\section{Final Remarks}\label{sec5}

\setcounter{section}{5}
\setcounter{equation}{0}\setcounter{theorem}{0}

We close hoping the reader has seen the connection between fractional calculus and recursion. The indefinite sum behaves quite neatly and we have only scratched the surface of its structure and uses. We state that this paper is sister to \cite{ref12}--on the differintegral and complex iterations. The two problems appear as an intrigue in iteration and show the power of fractional calculus and its familiarity and connection with iteration in general. For a more clear application of this, observe \cite{ref12}. We are satisfied with the brevity of this paper as it is meant to show a pure mathematical curiousity which follows through fractional calculus.

Although only briefly mentioned, our generalization of the Gamma function offers many questions. Can this functions be analytically continued to a meromorphic function for $q \in \mathbb{C}$ and $z \in \mathbb{C}$? What is its Weierstrass factorization in accordance to its multiplicative inverses zeroes? Explicitly, what is the Weierstrass factorization of $\frac{1}{\G_q(z)}$ in $z$ dependent on $q$? 

We are aware of more advanced techniques that apply to more complicated recursions. Many linear algebraic recursions in $T f(z) = f(z+1)$ can be solved by these methods. Notably we solved for complex values of the recursion $\frac{1}{1 - T^{-1}}$. We can do this for more complicated recursions by creating a similar isomorphisms between spaces that the Twisted Mellin transform induces. We can also perform fractional iterations of the recursion in a manner similar to how we produced the differsum. These results are very exotic and further capitulate the importance of the differintegral. 

\section{Acknowledgements}
The author is indebted to the University of Toronto where he is a student who makes constant use of their mathematics library, of which these results would've never come to fruition without.


\begin{thebibliography}{6}
\normalsize

\bibitem{rf2}
{Eric Delabaere, \textsl{Ramanujan's Summation}, Algorithms Seminar, (2001-2002).}
\bibitem{rf3}
{Chelo Ferreira, Jos$\acute{e}$ L. L$\acute{o}$pez, Ester Pérez Sinusía, \textsl{Incomplete gamma function for large values of their variables}, Advances in applied mathematics, (2004).}
\bibitem{rf9}
{Elias M. Stein and Rami Shakarchi, \textsl{Complex Analysis}, Princeton University Press, (2003).}
\bibitem{rf10}
{Masaaki Sugihara, \textsl{Justification of a formal derivation of the Euler-Maclaurin summation formula}, Analytic Extension Formulas and their Applications, Volume 9, (2001), pp 251-261.}
\bibitem{rf11}
{Zuoquin Wang, \textsl{The Twisted Mellin Transform}, arXiv, (2007).}
\bibitem{ref12}
{James Nixon, \textsl{Complex Iterations and Bounded Analytic Hyper-Operators}, arXiv, (2015).}
\end{thebibliography}
\end{document}